\documentclass[11pt, reqno]{amsart}

\usepackage{url}
\usepackage{hyperref}

\usepackage{amssymb,amsmath,amscd,graphicx,
latexsym,amsthm}
\usepackage{amssymb,latexsym,amsmath,amscd,graphicx}
\usepackage[mathscr]{eucal}
 \input xy
 \xyoption{all}
\def\ZZ{{\mathbb Z}}

\def\CC{{\mathbb C}}

\def\OO{{\mathcal O}}
\def\R{{\mathbf R}}
\def\L{\mathbf{L}}

\def\F{\mathcal{F}}
\def\E{\mathcal{E}}
\def\G{\mathcal{G}}
\def\H{\mathcal{H}}

\def\cP{\mathcal{P}}
\def\Pic0{\mathrm{Pic}^0}

\def\Aut0{\mathrm{Aut}^0}

\def\DD{{\mathbf{D}}}
\def\R{{\mathbf{R}}}

\def\*{{\underline *}}

\def\Supp{\mathrm{Supp}}

\def\cP{\mathcal{P}}
\def\cQ{\mathcal{Q}}
\def\cG{\mathcal{G}}

\def\Alb{\mathrm{Alb}\,}

\theoremstyle{plain}

\newtheorem{theorem}{Theorem}[subsection]

\newtheorem{proposition/example}[theorem]{Proposition/Example}
\newtheorem{proposition}[theorem]{Proposition}

\newtheorem{corollary}[theorem]{Corollary}
\newtheorem{lemma}[theorem]{Lemma}

\newtheorem{claim}[theorem]{Claim}

\theoremstyle{definition}
\newtheorem{definition}[theorem]{Definition}

\newtheorem{remark}[theorem]{Remark}

\newtheorem{conjecture/question}[theorem]{Conjecture/Question}

\newtheorem{remark/definition}[theorem]{Remark/Definition}
\newtheorem{notation/assumptions}[theorem]{Assumptions/Notation}

\numberwithin{equation}{section}

 \theoremstyle{remark}

\begin{document}
\title{Irregular fibrations of derived equivalent varieties}
 \author[ F. Caucci and L. Lombardi]{ Federico Caucci and  Luigi Lombardi}

\address{F. Caucci \\Dipartimento di Matematica, Universit\`a  di Roma Tor Vergata, Via della Ricerca
Scientifica 1, 00133 Roma, Italy}
 \email{caucci@mat.uniroma2.it }
\thanks{FC was partially supported by the ERC Consolidator Grant
ERC-2017-CoG-771507-StabCondEn, by the MUR Excellence Department
Project MatMod@TOV (2023-2027) awarded to the Department of Mathematics
of the University of Rome Tor Vergata, and by by the PRIN 2022 ``Moduli
spaces and Birational Geometry''. He is a member of GNSAGA-INdAM}

\address{L. Lombardi\\Dipartimento di Matematica ``Federigo Enriques'', 
             Universit\`a degli Studi di Milano, via Cesare Saldini 50, 20133 Milano, Italy}
\email{luigi.lombardi@unimi.it}
\thanks{LL was  partially supported by GNSAGA-INdAM, PSR Linea 4 of Universit\`a degli Studi di Milano, PRIN 2020: ``Curves, Ricci flat Varieties and their Interactions'' and PRIN 2022: ``Symplectic varieties: their interplay with Fano manifolds and derived categories''.}

\begin{abstract}
  We study the behavior of   irregular fibrations of a variety under derived equivalence of its bounded derived category. 
In particular we prove the derived invariance of the existence of  an irregular fibration over a variety of general type,  extending 
the case of irrational pencils onto curves of genus $g\geq 2$. 
We also prove that  a derived equivalence of such fibrations induces a derived equivalence between their general fibers.
\end{abstract}

\maketitle

\section{Introduction}\label{I0}

In this paper we investigate the invariance of
  irregular fibrations under derived equivalence. 
	An \emph{irregular fibration} is  a surjective morphism with connected fibers
from a smooth projective variety 
 onto a normal projective variety of positive dimension 
 admitting  a desingularization of maximal Albanese dimension.\footnote{This means that the Albanese map of this smooth model is generically finite onto its image. This property does not depend on the chosen desingularization.}
Two  smooth  projective  complex varieties $X$ and $Y$ are  \emph{derived equivalent} if there exists an  equivalence  of triangulated categories
$\Phi \colon \DD(X) \xrightarrow{\sim} \DD(Y)$  between their bounded derived categories of coherent sheaves. 
  The   theorem below is the main result of the paper. It concerns the derived invariance of irregular fibrations $f\colon X \to V$
    onto varieties  of general type, \emph{i.e.} such that  one (and hence any) resolution of singularities of $V$ is of general type. 
    These fibrations can  be regarded as a higher-dimensional 
analogue of the notion of irrational pencils over smooth curves of genus $g\geq 2$.  
  It turns out  that the mere existence of an irregular  fibration 
  imposes quite strong  restrictions  on the geometry of  Fourier--Mukai partners.
  
		\begin{theorem}\label{Thmfibrintro}
		Suppose $\DD(X) \simeq \DD(Y)$ and  that  
		$X$ carries an irregular fibration $f \colon X \to V$ such that  $V$  is of general type. Then:
\begin{itemize}
\item[(i)] 	
 $Y$  admits an irregular fibration $h \colon Y \to W$  
		such that $W$  is birational to $V$;  \\
		
\item[(ii)]		The general fibers of $f$ and $h$ are derived equivalent; \\
		
\item[(iii)]		If the (anti)canonical  line bundle of the general fiber of $f$ is big,
 then  $X$ and $Y$ are $K$-equivalent.

\end{itemize}		
		\end{theorem}

Recall that two varieties $X$ and $Y$ are  \emph{$K$-equivalent} if there exists a third smooth projective variety $Z$ and birational morphisms
\[
\xymatrix{
     &Z \ar[dl]_{p_X} \ar[dr]^{p_Y}  & \\
X  & &Y 
}
\]
such that $p_X^*\omega_X \simeq p_Y^*\omega_Y$. An important aspect of the $K$-equivalence relation 
is that it preserves 
the   Hodge numbers, as proved by Kontsevich's motivic integration theory.

 Theorem \ref{Thmfibrintro} provides  generalizations of several previously known results. 
First of all  
$(iii)$ above 
should be seen as a relative version of
 Kawamata's birational reconstruction theorem \cite{kawamata:dequivalence}.\footnote{Actually, by 
 the first works of Kawamata, Koll\'ar and     Viehweg on the Iitaka Conjecture,
if in Theorem \ref{Thmfibrintro} $(iii)$  the general fiber of $f$ is of general type, then $X$ itself is 
of general type  because $V_X$ is so (see, \emph{e.g.}, \cite[Theorem 1.2.9 and Problem 1.1.2]{fujino:iitaka}). 
However, we do not use this result here. Our approach  is self-contained and 
moreover it also works  \emph{verbatim} when the anticanonical line bundle of the general fiber of $f$ is big.}
Moreover,  it
    generalizes \cite[Theorem 6]{lombardipopa:conj} where the case
  of irregular fibrations onto curves of genus $g\geq 2$ was considered.

If one restricts to somewhat more specific irregular fibrations, one obtains a stronger result. Namely, assume that, beyond being of general type,  $V$ admits a morphism $c_{V} \colon V \to \Alb \widetilde{V}$ which is finite onto its image and such that the composition $\widetilde{V} \to V \xrightarrow{c_{V}} \Alb \widetilde{V_X}$ equals the Albanese map of a desingularization $\widetilde{V}$ 
(\emph{cf}. \cite[\S 3.2]{lombardi:fibrations}). 
Note that this is precisely what happens when $\dim V = 1$. 
\begin{proposition}\label{Propfibrintro}
Let $\DD(X) \simeq \DD(Y)$. 
Under the above assumption, if $\omega_X$ or $\omega_X^{-1}$ is $f$-ample, then $X$ is isomorphic to $Y$.
\end{proposition}

We now present the second main result of the paper which generalizes  \cite[Theorem 1]{lombardi:fibrations}.
We say that 
two irregular  fibrations
 $f_1 \colon X \to V_1$  and $f_2 \colon X \to V_2$ of a variety $X$  are \emph{equivalent} if there exists a birational map 
 $\sigma \colon V_1 \dashrightarrow  V_2$ 
such that $f_2=\sigma \circ f_1$.

\begin{theorem}\label{introbij}
Suppose $\DD(X) \simeq \DD(Y)$.
There exists a base-preserving bijection between  the sets of equivalence classes of irregular fibrations of $X$ and   $Y$ onto 
normal projective varieties  of general type. 
\end{theorem}
 We refer the reader to Theorem \ref{thm:corr} for the proof of Theorem \ref{introbij} and  
a more precise statement. See also Remark \ref{lastrmk}.

To \emph{any} fibration $f	\colon X \to V$ onto a normal projective variety $V$ there 
 is attached an abelian subvariety $\widehat{B_V}$ of $\Pic0 X$  as follows. 
Let $\widetilde{f} \colon \widetilde X \to \widetilde V$ be a \emph{non-singular representative} of  
the  fibration $f$, namely   a commutative diagram
\begin{equation}\label{diagr:Vintro}
  \xymatrix{\widetilde X \ar[d]^-{\widetilde f}\ar[r]^{\rho} & X  \ar[d]^-{f}\\
\widetilde{V}\ar[r]^{\theta} & V}  
\end{equation}
where $\rho$ and $\theta$ are birational morphisms from smooth projective varieties and $\widetilde f$ is a fibration.
By noting that the push-forward  map $\rho_* \colon \Pic0 \widetilde X \to \Pic0 X$  is an isomorphism, 
we define  the abelian subvariety
\begin{equation*}
\widehat{B_V}  \;  \stackrel{{\rm def} }{=}  \; \rho_*\widetilde f^* \Pic0 \widetilde V \; \subset \; \Pic0 X.
\end{equation*}
It is easy to check that $\widehat{B_V}$ is well-defined, \emph{i.e.}, it does not depend on the choice of the non-singular representative.  
What happens is 
that, if $\widetilde V$ is of general type and of maximal Albanese dimension as in Theorem \ref{Thmfibrintro}, then  $\widehat{B_V}$ is a \emph{Rouquier-stable subvariety} with respect to any  exact equivalence $\DD(X) \simeq \DD(Y)$ (\emph{cf}. Lemma \ref{lemmagentypeRunv}).
The notion of  Rouquier-stable subvarieties was  
introduced  in \cite{calopa:forth} in order to study the derived invariance of certain relative canonical rings. 
Briefly, it refers to abelian subvarieties  $\widehat{B_X} \subset \Pic0 X$ that are mapped isomorphically  via the Rouquier isomorphism \eqref{intro-rouquier}
to abelian subvarieties $\widehat{B_Y}$ of $\Pic0 Y$.
We  refer the reader to \S\S \ref{sectionRunmixed} and \ref{F0} for the definition and  main properties of Rouquier-stable subvarieties. 
The above fact allows us to apply the general results of \S \ref{F0} to the setting  of 
 Theorems \ref{Thmfibrintro} and \ref{introbij}. 
The proof of $(ii)$  also builds  on the latest relativization technique  for the kernel \cite{lieblich+olsson:2022}.

Finally, we note that in fact we   prove  slightly more general results than those of 
Theorem \ref{Thmfibrintro},  although in a little less geometric settings
(see Theorems \ref{chiposi} and \ref{relkawintro}). 
For instance, we record the following particular case of Theorem \ref{relkawintro}.
\begin{corollary}\label{corintro}
Let $X$ and $Y$ be derived equivalent  varieties. Assume   
that $\chi(X, \OO_X) \neq 0$. 
If  the (anti)canonical line bundle of the general fiber of the Albanese map of $X$ is big, then $X$ and $Y$ are $K$-equivalent. 
\end{corollary}

In another direction,  even if the base $V$  of an irregular fibration as in \eqref{diagr:Vintro} is not of general type, then in any case $\widehat{B_V}$ contains a certain Rouquier-stable subvariety (namely, $\Pic0$ of the base of the Iitaka fibration of $V$),  leading to  the following result extending Theorem \ref{Thmfibrintro}(i).
\begin{theorem}\label{thmnotgentype}
Suppose  $\DD(X) \simeq \DD(Y)$. If  $X$ admits an irregular fibration $f \colon X \to V$, then
there exists 
 a fibration $h \colon Y \to W$ of $Y$ onto a normal projective variety $W$ 
 which is birational to the base of the Iitaka fibration of $V$. In particular, we have	
 	 $\dim W = \mathrm{kod} (V)$  and   any smooth model of $W$ is of maximal Albanese dimension.
\end{theorem}

This quite satisfactorily answers 
  the problem of understanding in which manner an arbitrary  irregular fibration of a given variety varies under derived equivalence of its bounded derived category.
Moreover, in \cite{popa:conj} Popa conjectured the derived invariance of  non-vanishing canonical loci (see also \S \ref{cohsupploci} below) figuring out that the geometric meaning of his conjecture is 
 that derived equivalent varieties should have the same type of fibrations over
lower dimensional irregular varieties, this   allowing for more geometric tools in the study
of Fourier--Mukai partners. 
Popa proved a version of this principle assuming his conjecture (see \cite[Corollary 3.4]{popa:conj}).
Here we notice that,  although Popa's conjecture is still open, 
  our techniques allow to get an unconditional  proof of  his result: 
\begin{theorem}
Suppose $\DD(X) \simeq \DD(Y)$. If  $X$ admits a fibration $f \colon X \to Z$ onto a normal projective $m$-dimensional variety $Z$  whose Albanese map\footnote{Actually the Albanese map of a desingularization of $Z$.}  is not surjective, then $Y$ admits an irregular fibration $h \colon Y \to W$ onto a  variety  of general type $W$,  with $0 < \dim W \leq m$.
\end{theorem}
\begin{proof} If the Albanese map of (a desingularization of) $Z$  is not surjective, then its image  admits a fiber bundle structure $g$ onto a positive dimensional variety  of maximal Albanese dimension and of general type   \cite[Theorem 10.9 and Corollary 10.5]{ueno:book}. By taking the Stein factorization of the composition  $g \circ a_{Z}$, we see that    
 $Z$ admits a fibration $g'$ onto a normal projective variety of general type, maximal Albanese dimension and positive dimension $\leq m$.  So 
Theorem \ref{Thmfibrintro}(i) (and its proof)  may apply to the composition $g' \circ f$.
\end{proof}

\noindent  \textbf{Acknowledgements.}
We are grateful to Beppe Pareschi for his insights and advice, and to Mihnea Popa for fruitful conversations
regarding the statements of Theorem \ref{Thmfibrintro}.
LL thanks Zhi Jiang for his hospitality at  the 
Shanghai Center for Mathematical Sciences (Fudan University)
where part of this work was carried out.
We are also grateful to the 
 anonymous referees whose  carefully reading  and useful comments (especially concerning \S 3) allowed us to improve the paper.

\noindent \textbf{Notation.}
 Our ground field is the field of complex numbers $\CC$. 
 A variety means an irreducible smooth projective variety, unless otherwise stated. 
A fibration is a surjective morphism of normal projective varieties with connected fibers.
  If $X$ is a variety, we denote by $\DD(X) := \DD^b \big(\mathcal{C}oh (X) \big)$ the bounded derived category of coherent sheaves on $X$.  
 The Albanese map of $X$
is denoted by $a_X \colon X \to \Alb X$. The irregularity of $X$ is $q(X):= h^{1, 0}(X) = \dim \Alb X$.
If we fix a Poincar\'e line bundle $\cP$ on $X \times \Pic0 X$, 
we denote by $P_\alpha := \cP |_{ X \times \{ \alpha \} }$  
the line bundle  parameterized by the  point $\alpha\in\Pic0 X$.
%
Given a morphism of abelian varieties $\varphi \colon A\rightarrow B$,  we denote by 
$\widehat \varphi \colon \widehat B \rightarrow \widehat A$  the dual morphism of $\varphi$.

\section{Rouquier-stable subvarieties}\label{sectionRunmixed}

Let $X$ and $Y$ be smooth  projective  complex varieties and let 
$\Phi \colon \DD(X) \xrightarrow{\sim} \DD(Y)$ be an  exact equivalence  between the derived categories of $X$ and $Y$. 
The equivalence $\Phi$ induces  a functorial  isomorphism of algebraic groups
 \begin{equation}\label{intro-rouquier}
  \varphi \colon \Aut0 X\times \Pic0 X \buildrel\sim\over\longrightarrow  \Aut0 Y\times \Pic0 Y
 \end{equation}
known as the \emph{Rouquier isomorphism} \cite{rouquier:automorphisms}. This isomorphism 
has been employed by Orlov in \cite{orlov:eqAV} in order to classify derived equivalences of abelian varieties. Moreover, it   
plays a crucial role in Popa--Schnell's proof of the derived invariance of the irregularity \cite{popa+schnell:invariance}.
Other applications are contained in
 \cite{lombardi:invariants}, \cite{lombardipopa:conj}, \cite{caucci+pareschi:invariants} and  \cite{lieblich+olsson:2022}.

The main difficulty in dealing with the Rouquier isomorphism is that, in general,  it 
does not respect the factors.
Namely,  there exist equivalences such that 
$$\varphi \big(  \{ {\rm id}_X \} \times \Pic0 X \big) \neq  \{ {\rm id}_Y \} \times \Pic0 Y.$$ 
For instance, this is the case of 
the Fourier--Mukai--Poincar\'e transform between an abelian variety and its dual.
Quite naturally, one is led to consider \emph{Rouquier-stable subvarieties}   as introduced in 
 \cite{calopa:forth}.
\begin{definition} 
 An abelian subvariety $\widehat{B_X} \subset \Pic0 X$ is  \emph{Rouquier-stable} (with respect to the equivalence $\Phi$),  
 if the induced Rouquier isomorphism \eqref{intro-rouquier} satisfies
\[
\varphi \big( \{ {\rm id}_X \} \times \widehat{B_X} \, \big) \subseteq \{ {\rm id}_Y \} \times \Pic0 Y.
\]
We denote by $\widehat{B_Y}$ the 
abelian subvariety 
$$ \widehat{B_Y} \; := 
p_{\Pic0 Y} \big( \varphi \big( \{ {\rm id}_X \} \times \widehat{B_X} \, \big) \big)$$ 
 of $\Pic0 Y$, where 
$p_{\Pic0 Y} \colon \Aut0 Y \times \Pic0 Y \to \Pic0 Y$ is the projection onto the second factor. 
By a slight abuse of notation 
we simply  write $\widehat{B_Y} = \varphi (\widehat{B_X})$.
\end{definition}
We refer the reader to Subsection \ref{secex} for several examples  of Rouquier-stable subvarieties.

A Rouquier-stable subvariety $\widehat{B_X}\subseteq \Pic0 X $ induces two 
morphisms.  The first $b_X \colon X \to B_X$
is   given as the composition
\begin{equation*}
\xymatrix{
 X \ar@/^2pc/[rr]^{b_X} \ar[r]^{a_X} & \Alb X \ar[r]^{ \;\;\pi_{B_X}} &  B_X
 } 
\end{equation*}
 of the Albanese map $a_X\colon X \to \Alb X$ with the dual morphism $\pi_{B_X}$  
of the inclusion $\widehat{B_X}\subseteq \Pic0 X$.
The second $b_Y \colon Y \to B_Y$
 is defined similarly as the composition
\begin{equation*}
\xymatrix{
 Y \ar@/^2pc/[rr]^{b_Y} \ar[r]^{a_Y} & \Alb Y \ar[r]^{ \;\;\pi_{B_Y}} &  B_Y.
 } 
\end{equation*} 
 We refer to  the morphisms $b_X \colon X \to B_X$ and $b_Y \colon Y \to B_Y$ 
 as   a  \emph{(pair of) Rouquier-stable morphisms}. By taking the Stein factorization, we have 
 commutative diagrams
 \begin{equation*}\label{stein} 
\xymatrix{ X \ar[r]^{a_X\;\;\;} \ar[rd]^{b_X} \ar[d]^{s_X} & \Alb X \ar[d]^{\pi_{B_X}} \\ 
 X' \ar[r]^{b'_X}   & B_X} \qquad \qquad
 \qquad \xymatrix{ Y \ar[r]^{a_Y\;\;\;} \ar[rd]^{b_Y} \ar[d]^{s_Y} & \Alb Y \ar[d]^{\pi_{B_Y}} \\ 
 Y' \ar[r]^{b'_Y} & B_Y} 
 \end{equation*} 
where $s_X \colon X \to X'$  and $s_Y\colon Y \to Y'$ 
are fibrations onto normal projective varieties and $b_X'\colon X' \to B_X$ and $b_Y' \colon Y' \to B_Y$ 
are finite morphism onto their images. A result of \cite{calopa:forth} shows that the finite components of these Stein factorizations are isomorphic. More
precisely, 
there exists an isomorphism $\psi \colon Y' \xrightarrow{\sim} X'$ such that the diagram 
\begin{equation*}\label{xy2-intro}
\xymatrix{X' \ar[d]^-{b_X'} &Y' \ar[l]_-{\psi} \ar[d]^-{b_Y'}\\
B_X  &B_Y \ar[l]_-{\widehat{\varphi}}}
\end{equation*}
 is commutative, where $\widehat{\varphi}$ is the dual isomorphism. In \S \ref{F0} we will recall this fact and, 
building on \cite{lieblich+olsson:2022},
we show that  the general fibers of  $s_X$ and $s_Y$ are derived equivalent.

\subsection{Examples}\label{secex}

In this subsection we present a few  examples of Rouquier-stable subvarieties. 
Let us fix   an equivalence $\Phi \colon \DD(X) \to \DD(Y)$ of triangulated categories
and  let 
$\varphi \colon \Aut0 X \times  \Pic0 X   \to   \Aut0 Y \times  \Pic0 Y$  be the induced Rouquier isomorphism. 
We point out that, aside from \ref{SFE}, all the examples presented below are intrinsically Rouquier-stable, \emph{i.e.}, they are stable with respect to any equivalence.

\subsubsection{The trivial example}
The subset $\{ \, \hat{0} \, \}\subset \Pic0 X$    is Rouquier-stable.
The induced pair of  Rouquier-stable morphisms are 
the   constant maps $X \to \{0\}$ and $Y\to \{0\}$.

\subsubsection{A numerical condition}\label{subs:num}
Let $a(X) := \dim \Alb ( \Aut0 X)$. If $q(X) > a(X)$, then 
there exists a 
Rouquier-stable subvariety of positive dimension of $\Pic0 X$. For a proof of this fact we adopt the terminology of \cite[pp. 532-533]{popa+schnell:invariance}.
Set  $G_X = \Aut0(X)$.
Then $\dim ({\rm ker}(\pi)_0) = \dim \big( {\rm ker}(\Pic0(X) \to \widehat A)_0 \big)  = q(X) - \dim \widehat A \geq q(X) - a(X) >0$ so that 
${\rm ker}(\pi)_0$ is 
an abelian variety of positive dimension, which must be  Rouquier-stable.

 \subsubsection{Affine automorphism group} \label{subs:aut}
This is a special case of the above \eqref{subs:num}.   
Following \cite[p. 534]{popa+schnell:invariance},  
$\Aut0 X$ is an affine algebraic group if and only if $\Aut0 Y$ is so. In this case, since $\Pic0 X$ is projective and the irregularity is a derived invariant  (\emph{cf}. \cite[Corollary B]{popa+schnell:invariance}), one has 
\begin{equation}\label{invapic}
\varphi \big( \{ {\rm id}_X \} \times  \Pic0 X  \big) =  \{ {\rm id}_Y \} \times \Pic0 Y,
\end{equation}
so that $\Pic0 X$ is  Rouquier-stable.\footnote{Proof: Since $\Aut0 Y$ is affine,  the composition
$\{ \rm{id}_X \} \times \Pic0 X \xrightarrow{\varphi} \Aut0 Y \times \Pic0 Y \rightarrow \Aut0 Y$
is constant.} 
The induced pair of Rouquier-stable morphisms are the 
Albanese maps $a_X$ and $a_Y$ themselves.
Instances of varieties with  affine automorphism group ${\rm Aut}^0(-)$ are  varieties 
with non-vanishing Euler characteristic  $\chi(X, \OO_X) \neq 0$ \cite[Corollary 2.6]{popa+schnell:invariance},
 and varieties  with big (anti)canonical line bundle (see, \emph{e.g.}, \cite[Proposition 2.26]{brion:notes} for a detailed proof of this folklore result. 
 Besides,  the automorphism group ${\rm Aut}(-)$ of a  variety of general type is  finite by a classical result of Matsumura \cite[Corollari 2]{matsumura:gentype}).

 \subsubsection{Strongly filtered equivalences}\label{SFE}
 In the   paper \cite{lieblich+olsson:birational}, the authors   introduce a 
 notion of   equivalence 
  called \emph{strongly filtered}.\footnote{An equivalence
 is   \emph{strongly filtered} if  it 
preserves the codimension filtration on the numerical Chow ring, together  with the Hochschild--Kostant--Rosenberg  filtrations on 
the  Hochschild homology and cohomology. } 
For this type of equivalence the
  formula \eqref{invapic} continues to hold. 
As suggested in \cite{popa+schnell:invariance}  and \cite{lieblich+olsson:birational}, 
the level of mixedness of the Rouquier isomorphism could be interpreted as 
a measure of the complexity of a derived equivalence from the point of view of birational geometry. 
For instance, in \cite{lieblich+olsson:birational} it is proved  that 
a strongly filtered equivalence of smooth 
projective threefolds with positive irregularity 
induces a birational isomorphism.

\subsubsection{Cohomological support loci}\label{cohsupploci}	
	 Given a coherent  sheaf $\F$ on a variety $X$, the \emph{cohomological support loci attached to} $\F$ are  the algebraic closed subsets 
\begin{equation*}\label{GL-set}
V^i ( X , \F ) \,  = \,   \{ \, \alpha\in \Pic0 X \; \big|  \; H^i(X,\F\otimes P_\alpha) \neq 0 \, \}.
\end{equation*}
Let us denote by $V^i(X, \F)_0$  the union of the irreducible components of $V^i(X, \F)$ passing through the origin.
By \cite[Claim 3.3]{lombardi:invariants}, one has
\begin{equation}\label{nonvanloci}
\varphi \big( \{{\rm id}_X\} \times V^i(X , \Omega_X^j \otimes \omega_X^{\otimes m} )_0 \big) \subseteq \{ {\rm id}_Y \} \times \Pic0Y
\end{equation}
for all $i, j \geq 0$ and $m \in \ZZ$.
In particular, any abelian subvariety of $\Pic0 X$ that is contained in some $V^i(X , \Omega_X^j \otimes \omega_X^{\otimes m})_0$  is 
Rouquier-stable. Moreover,
since $\varphi$ is an isomorphism of algebraic groups, the abelian subvarieties of $\Pic0 X$ generated by the loci 
$V^i(X , \Omega_X^j \otimes \omega_X^{\otimes m})_0$ (or by any of their irreducible components) are Rouquier-stable.
It is also known  by \cite[Proposition 3.1]{lombardi:invariants}
that
the Rouquier isomorphism preserves the full loci $V^0(X , \omega_X^{\otimes m})$, namely
\begin{equation}\label{invV0}
\varphi \big( \{ {\rm id}_X \} \times V^0( X, \omega_X^{\otimes m} ) \big) = \{  {\rm id}_Y \} \times V^0( Y , \omega_Y^{\otimes m} ),\quad \forall m\in \ZZ.
\end{equation}
Note that the same behavior is expected to hold for the loci $V^i(X ,  \omega_X )_0$
for any  $i$
  \cite[Conjecture 11]{lombardipopa:conj}.

\subsubsection{The Albanese--Iitaka morphism}
Suppose that the  Kodaira dimension of $X$ is non-negative.
We can choose a smooth birational modification $\widetilde{X} \rightarrow X$ such that the Iitaka fibration is represented by a morphism 
$f_{\widetilde X} \colon \widetilde{X} \rightarrow Z_X$ with $Z_X$ smooth. More concretely, we have a 
  commutative diagram
\begin{equation*}\label{diag}
\xymatrix{
\widetilde X \ar@/^2pc/[rr]^{a_{\widetilde X}} \ar[r] \ar[rd]_{f_{\widetilde X} } & X \ar[r]^{a_X} \ar@{-->}[d] \ar[dr]^{c_X} & \Alb X \ar[d]^{\pi_X} \\
& {Z_X} \ar[r]^{a_{Z_X}} & \Alb Z_X
}
\end{equation*}
where $a_X, a_{\widetilde X}$ and $a_{Z_X}$ are  Albanese maps, and  $\pi_X$ is a fibration of abelian varieties 
induced by $f_{\widetilde X}$ (\emph{cf}. \cite[Lemma 11.1(a)]{hps:metrics}). 
Note that  $\Alb Z_X$,  $\pi_X$ and  $c_X = \pi_X \circ a_X$  only depend on $X$,
 and not on the modification $\widetilde X$ we fixed. 
In \cite{calopa:forth}, the morphism $c_X$ is called the \emph{Albanese--Iitaka morphism} of $X$.
It follows from \eqref{invV0} that
 the Rouquier isomorphism acts  as
$$\varphi \big( \{ {\rm id}_X \} \times  \widehat{\pi}_X(\Pic0 Z_X) \big) = \{ {\rm id}_Y \} \times  \widehat{\pi}_Y(\Pic0 Z_Y)$$
(see  \cite[Proof of Lemma 3.4]{caucci+pareschi:invariants}).
 In particular,
$\widehat{\pi}_X(\Pic0 Z_X)$ is a  Rouquier-stable subvariety  and the Albanese--Iitaka morphisms $c_X$ and $c_Y$ are Rouquier-stable.
This fact was already noted (and used) in \cite{caucci+pareschi:invariants}  and  moreover it is particularly useful in \cite{calopa:forth}.

\subsubsection{Fibrations over varieties of general type}
Let $f\colon X \to V$ be an irregular fibration onto a normal projective  variety of general type. 
By keeping notation as in \eqref{diagr:Vintro},
we will show in Lemma \ref{lemmagentypeRunv} that the abelian subvariety $\rho_* \widetilde f^* \Pic0 \widetilde V
\subset \Pic0 X$ is Rouquier-stable. In particular, the 
abelian subvarieties attached to $\chi$-positive fibrations considered in 
\cite{lombardi:fibrations} are Rouquier-stable (\emph{cf}. \cite[Remark 14]{lombardi:fibrations}).

\section{The Stein factorization of a  Rouquier-stable morphism}\label{F0}

In this section we study the effects of
 the existence of a non-trivial Rouquier-stable subvariety.
Informally speaking, one such  subvariety   turns a derived equivalence into a  relative equivalence, at least generically.

Let $\Phi \colon \DD(X) \to \DD(Y)$ be an equivalence of triangulated categories and let $p\colon X \times Y \to X$ and $q\colon X \times Y \to Y$
be the natural projections onto the first and second factor, respectively.
By Orlov's representability theorem \cite[Theorem 2.2]{orlov:repr}, there exists an object $\E$ in $\DD(X\times Y)$ that is unique
up to isomorphism and such that  $\Phi(-) \simeq \Phi_{\E} (-) := \R q_* \big( p^* (-) \stackrel{\L}{\otimes} \E \big)$. 
We denote by $\varphi_{\E}$ the Rouquier isomorphism induced by $\Phi_{\E}$. 
\begin{theorem}\label{step1} 
Let $\Phi_{\E} \colon \DD(X) \to \DD(Y)$ be an equivalence and let $\widehat{B_X}\subset \Pic0 X$ be a Rouquier-stable subvariety.
Moreover, let $b_X \colon X \to B_X$ and $b_Y \colon Y \to B_Y$ be the induced pair
of Rouquier-stable morphisms.
By considering the Stein factorizations of $b_X$ and $b_Y$,  we have the following commutative diagrams
 \begin{equation}\label{xy}
 \xymatrix{ X \ar[r]^{a_X\;\;\;\;\;} \ar[rd]^{b_X} \ar[d]^{s_X} & \Alb X \ar[d]^{\pi_{B_X}} \\ 
 X' \ar[r]^{b'_X}   & B_X} \qquad \qquad
 \qquad \xymatrix{ Y \ar[r]^{a_Y\;\;\;} \ar[rd]^{b_Y} \ar[d]^{s_Y} & \Alb Y \ar[d]^{\pi_{B_Y}} \\ 
 Y' \ar[r]^{b'_Y} & B_Y} 
 \end{equation} 
where $s_X$ and $s_Y$ are surjective morphisms with connected fibers, and $b_X'$ and $b_Y'$ are finite morphisms onto their images. Then:
\begin{itemize}
\item[(i)] There exists  
an isomorphism $\psi \colon Y^\prime \stackrel{ \sim }{\rightarrow} X^\prime$ of normal projective  varieties such that 
the kernel $\E$ is set-theoretically supported on the fiber product $X \times_{Y'} Y$ defined as follows:
\begin{equation*}\label{xyfiberp}
\xymatrix{
X \times_{Y'} Y \ar[r] \ar[d]  & X \ar[d]^{\psi^{-1} \circ s_X}\\
Y \ar[r]^{s_Y} &  Y'. }
\end{equation*}

\item[(ii)] The finite parts of the Stein factorizations 
  are isomorphic, \emph{i.e.}
 the following diagram commutes
\begin{equation*}\label{xy2'}
\xymatrix{X' \ar[d]^{b_X'} &Y' \ar[l]^{\psi}_\sim\ar[d]^{b_Y'}\\
B_X & B_Y. \ar[l]^{\widehat{\varphi_\E}}_\sim}
\end{equation*}

\item[(iii)]  The fibers $s_X^{-1} \big( \psi(y') \big)$ and $s_Y^{-1}(y')$ are derived equivalent for   $y'$ general in $Y'$.\\

\end{itemize}
\end{theorem}

\begin{proof} 
The proofs of $(i)$ and $(ii)$ are given in \cite{calopa:forth} (see \S 8.1 of \emph{loc.\! cit.\!} and, especially, Theorem 8.1.1)  and follow the general strategy of \cite[Theorem 1]{lombardi:fibrations}.
We just recall the main points of  the proof of $(i)$ for reader's convenience.
Denote by  $$p'\colon X'\times Y' \to X' \quad \mbox{ and }\quad q' \colon X'\times Y' \to Y'$$
the natural projections. 
Let ${\rm Supp}(\E): = \bigcup_j \mathcal {\rm Supp} \big(\H^j(\E) \big) \subseteq X \times Y$ 
be the support of $\E$, equipped with the reduced scheme structure. 
Then the projections $p' \colon (s_X \times s_Y)(\Supp (\E)) \to X'$ and $q' \colon (s_X \times s_Y)(\Supp(\E)) \to Y'$ 
have finite fibers. Moreover, they are surjective with connected fibers. In other words,
 $(s_X \times s_Y) \big({\rm Supp}(\E) \big)$  dominates isomorphically both $X'$ and $Y'$.
Hence the map $\psi := ( p'\circ q'^{-1} ) \colon Y' \to X'$ is  an isomorphism and 
  $(s_X \times s_Y) \big( {\rm Supp}(\E) \big) = {\rm Graph}(\psi)$. In particular, we have 
\begin{equation*}\label{eq:suppfp}
 {\rm Supp}(\E)  \subset (s_X \times s_Y)^{-1} \big( {\rm Graph} (\psi) \big)  = X \times_{Y'} Y.
\end{equation*}

 In order to prove $(iii)$ we 
denote by $\tau \colon U \hookrightarrow Y'$ a smooth open subvariety 
over which  both $s_Y$ and $\psi^{-1} \circ s_X$ are smooth  morphisms, and define 
the preimages
$$X_U : = (\psi^{-1} \circ s_X)^{-1}(U) \quad \mbox{ and } \quad Y_U : =  s_Y^{-1}(U).$$
By a slight abuse of notation, we continue to denote by $\psi^{-1} \circ s_X$ and $s_Y$ the two 
restrictions $(\psi^{-1} \circ s_X)|_{X_U}  \colon X_U \to U$ and $s_Y|_{Y_U} \colon Y_U \to U$, respectively. 
Moreover, we consider  the  closed subscheme
\begin{equation}\label{defl}
Z \, := \, X_U \times_U Y_U \; \stackrel{ \ell }{\hookrightarrow} \; X_U \times Y_U.
\end{equation} 
As $\E$ is set-theoretically supported on $X\times_{Y'} Y$, 
the derived restriction $k^*\E $  
is set-theoretically supported on $Z$,
where  $k\colon X_U\times Y_U \hookrightarrow  X\times Y$ is the inclusion map.
Denote by 
$$i\colon X_U \times Y \hookrightarrow  X\times Y \quad \mbox{and}\quad  j \colon  X_U \times Y_U \hookrightarrow X_U \times Y$$ 
 the open immersions so that $k = i \circ j$.

\begin{claim}\label{rel0}
The kernel $\E_U :=  k^*\E \in \DD(X_U \times Y_U )$ defines an equivalence of bounded derived categories
$$\Phi_{\E_U} \colon \DD ( X_U ) \to \DD ( Y_U ) $$ (the functor is well-defined as the support of $\E_U$ is proper over both 
$X_U$ and $Y_U$). 
\end{claim}

\begin{proof}
Let $n = \dim X$.
The claim is    proved in  
\cite[3.18]{lieblich+olsson:2022}. For reader's ease we reproduce here the argument.
Denote by ${\rm ad}(\E)$ the adjoint kernel:
$${\rm ad}( \E)  := \E^{\stackrel{\R}{\vee}} \otimes p^*\omega_X[n] \simeq \E^{\stackrel{\R}{\vee}} \otimes q^*\omega_Y[n]$$ in $\DD(X\times Y)$ (the superscript $\stackrel{\R}{\vee}$ denotes the derived dual).
By considering the Fourier--Mukai transform in the other direction
$\Psi_{{\rm ad}( \E)  } (-) := \R p_* (q^*(-) \stackrel{\L}{\otimes} {\rm ad}( \E)  ) \colon \DD(Y) \to \DD(X)$, 
one can check that $\Psi_{{\rm ad}( \E)  }$ is a quasi-inverse of 
$\Phi_{\E}$. 
By denoting by $p_{ij}$ the projections from $X\times Y\times X$ onto the $i$-th   and  $j$-th factors,  it follows 
$$\R p_{13*} ( p_{12}^* \E \stackrel{\L}{\otimes } p_{32}^* {\rm ad}(\E) ) \simeq \delta_{X*}\OO_X,$$ where
$\delta_{X} \colon X\hookrightarrow X\times X$ is the diagonal embedding. 
 Let $a_{ij}$ be the projection from   $ X_U \times Y_U \times X_U$ onto the $i$-th   and  $j$-th factors and 
set ${\rm ad}(\E)_U  :=  k^* {\rm ad}( \E)$.  Moreover let $\delta_{X_U}\colon X_U \hookrightarrow X_U\times X_U$ be the diagonal embedding of $X_U$.
By pulling-back  the above isomorphism under the 
open immersion $r\colon X_U\times X_U \to X\times X$, and by noting that
$i^* \E$ is supported on $X_U\times Y_U$ so that $i^*\E \simeq \R j_*\E_U$ and  $i^*{\rm ad} (\E) \simeq \R j_* {\rm ad}(\E)_U$   
(see \cite[p.45 (1.4.3.4)]{BBD:sheaves}  or \cite[proof of Lemma 36.6.2]{stacks-project}), we have 
the isomorphism 
$$ \R a_{13*} ( a_{12}^* \E_U \stackrel{\L}{\otimes } a_{32}^* {\rm ad}(\E)_U) \simeq \delta_{X_U*} \OO_{X_U}.$$ 
Similarly   we can prove
$$ \R a_{13*} ( a_{12}^* {\rm ad}(\E)_U   \stackrel{\L}{\otimes } a_{32}^*\E_U )\simeq \delta_{X_U*} \OO_{X_U}$$ 
and  that $\Phi_{\E_U}$ is an equivalence.
 \end{proof}

\begin{claim}\label{lastclaim}
The restricted kernel $\E_U$   
is isomorphic to a pushforward $\ell_*\mathcal C$ for some object $\mathcal C$ in $\DD(Z)$, where $\ell $ is defined in \eqref{defl}.
\end{claim}
Assuming the above Claim \ref{lastclaim} for a moment, we conclude the proof as follows. 
From the isomorphism $\E_U \simeq \ell_*\mathcal C$ 
we have that
$$\Phi_{\E_{U}} = \Phi_{\mathcal C}\, \colon \DD ( X_U  ) \to \DD ( Y_U  ) $$ 
is a \emph{relative} integral functor. As   showed in \cite[Propositions 2.15 and  2.10]{lrs:relative}, 
 the derived restriction
$$\mathcal{C}_u  := \mathcal{C}|_{(\psi^{-1} \circ s_X )^{-1}(u) \times s_Y^{-1}(u)   }$$ 
induces a derived equivalence 
$\Phi_{\mathcal{C}_u} \colon \DD \big(  s_X^{-1}  (\psi( u ) ) \big) \to \DD \big( s_Y^{-1} ( u ) \big)$ 
for any closed point $u \in U$   if 
 $\Phi_{\E_{U}}$ is an equivalence. Therefore,  we get $(iii)$.  
\end{proof}

\begin{remark}
The equivalence $\Phi_{\E_U} \colon \DD ( X_U  ) \to \DD ( Y_U  )$ is $U$\emph{-linear}
in the sense that for all $\F$  in   $\DD ( X_U  )$ and $\G$ in $\DD(U)$ there are  bifunctorial isomorphisms  
$$ \Phi_{\E_U}  \big( \F \stackrel{ \mathbf L }{\otimes } (\psi^{-1} \circ s_X)^* \G \big) \simeq 
  \Phi_{\E_U}  ( \F ) \stackrel{ \mathbf L }{\otimes } s_Y^* \G$$
 (\emph{cf}. \cite[Lemma 2.33]{ku:hpd}).
\end{remark}

\begin{proof}[Proof of Claim \ref{lastclaim}]
This is an application of the criterion Theorem 1.1 in \cite{lieblich+olsson:2022}. More precisely, 
we need to verify the  conditions  \eqref{condabc}  
below, in order to apply \cite[Theorem 1.1]{lieblich+olsson:2022} and hence to get our result.
In what follows we argue similarly to \cite[Proof of Lemma  4.11]{lieblich+olsson:2022}. 
Recall the commutative diagram
\[
\xymatrix{
   &X  \ar[dl]_-{b_X} \ar[d]^-{s_X}     &Y \ar[d]_-{s_Y} \ar[dr]^-{b_Y}  &  \\
B_X   &X' \ar[l]_-{b_X'} &Y' \ar[l]_-{\psi} \ar[r]^-{b_Y'} &B_Y \ar@/^2pc/[lll]_-{\widehat{\varphi}_{\E}}\, .	
}
\]
Let us define the morphisms
\[
u_X \colon Y' \to B_X \times Y', \quad p \mapsto (b_X'(\psi(p)), p) 
\]
and
\[
u_Y \colon Y' \to B_Y \times Y', \quad p \mapsto (b_Y'(p), p) = ((\widehat{\varphi_{\E}})^{-1}(b_X'(\psi(p))), p)\, . 
\]
\begin{lemma}
One has
\begin{equation}\label{LOeq0}
p_{12}^* (b_X \times \mathrm{id}_{Y'})^*({u_X}_*\OO_{Y'}) \otimes p_{13}^* \E \simeq 
p_{32}^* (b_Y \times \mathrm{id}_{Y'})^*({u_Y}_*\OO_{Y'}) \otimes p_{13}^* \E 
\end{equation}
in $\DD(X \times Y' \times Y)$, where we dropped the derived notation $\R$ and $\L$ for simplicity.
\end{lemma}
\begin{proof}
The isomorphism $(\widehat{\varphi_{\E}})^{-1} \times \varphi_{\E} \colon B_X \times \widehat{B_X} \to B_Y \times \widehat{B_Y}$ preserves  Poincar\'e line bundles, that is, $((\widehat{\varphi_{\E}})^{-1} \times \varphi_{\E})^*\cP \simeq \cQ$ where $\cP$ and $\cQ$ are normalized Poincar\'e line bundles on $B_Y \times \widehat{B_Y}$ and $B_X \times \widehat{B_X}$, respectively.
By  \cite[Theorem 1.1]{mukai:EQ2}, we have an equivalence of derived  categories
\[
\DD(\widehat{B_X} \times Y') \xrightarrow{\simeq} \DD(B_X \times Y'), \quad \cG \mapsto {p_2}_*(p_1^* \cG \otimes \overline{p}_{12}^*\cQ),
\]
where $\overline{p}_{12} \colon B_X \times \widehat{B_X} \times Y' \simeq  (\widehat{B_X} \times Y') \times_{Y'} (B_X \times Y') \to B_X \times \widehat{B_X}$. Moreover, the following diagram is commutative
\begin{equation}\label{LOeq1}
\xymatrix{
\DD(\widehat{B_X} \times Y') \ar[r]^-{\simeq} \ar[d]^-{(\varphi_{\E} \times \mathrm{id})_*} &\DD(B_X \times Y') \ar[d]^-{((\widehat{\varphi_{\E}})^{-1} \times \mathrm{id})_*} \\
\DD(\widehat{B_Y} \times Y') \ar[r]^-{\simeq} &\DD(B_Y \times Y')\, ,
}
\end{equation}
where the bottom equivalence is similarly defined,  using $\cP$ instead of $\cQ$.
In particular, there exists a unique object $\cG \in \DD(\widehat{B_X} \times Y')$ such that
\begin{equation}\label{LOeq2}
{u_X}_* \OO_{Y'} \simeq {p_2}_*(p_1^* \cG \otimes \overline{p}_{12}^*\cQ).
\end{equation}
Let us consider the  commutative diagram:
\[
\xymatrix{
        X \times \widehat{B_X} \times Y' \times Y \ar[r]^-{p_{134}} \ar[d]^-{p_{123}}  \ar@/^1.5pc/[rr]^-{p_{14}} &X \times Y' \times Y \ar[d] \ar[r] &X \times Y \\
				X \times \widehat{B_X} \times Y' \ar[r] \ar[d]^-{b_X \times \mathrm{id}_{\widehat{B_X} \times Y'}} &X \times Y' \ar[d]^-{b_X \times \mathrm{id}_{Y'}} & \\
 B_X \times \widehat{B_X} \times Y' = (\widehat{B_X} \times Y') \times_{Y'} (B_X \times Y')  \ar[d]^-{p_1} \ar[r]^-{p_2} &B_X \times Y' & \\
   \widehat{B_X} \times Y' & &			
}
\]
When needed, we identify $X \times \widehat{B_X} \times Y’\times Y$ with 
$X \times \widehat{B_Y} \times Y’\times Y$  via the isomorphism ${\rm id}_X\times \varphi_{\E} \times {\rm id}_{Y'} \times {\rm id}_Y$
in order to lighten notation.
From this and \eqref{LOeq2}, by using  flat base change and the  projection formula, 
 one obtains 
\begin{equation*}
\begin{split}
p_{12}^* (b_X \times {\rm id}_{Y'})^*({u_X}_*\OO_{Y'}) \otimes p_{13}^* \E &\simeq p_{12}^* (b_X \times {\rm id}_{Y'})^*({p_2}_*(p_1^* \cG \otimes \overline{p}_{12}^*\cQ)) \otimes p_{13}^* \E \\
&\simeq {p_{134}}_* (p_{123}^*(b_X \times {\rm id}_{\widehat{B_X} \times Y'})^*(p_1^* \cG \otimes \overline{p}_{12}^*\cQ)) \otimes p_{13}^* \E \\
&\simeq {p_{134}}_*(p_{23}^* \cG  \otimes p_{12}^*\cQ_X \otimes p_{14}^*\E)\, ,
\end{split}
\end{equation*}
where $\cQ_X := (b_X \times {\rm id}_{\widehat{B_X}})^* \cQ$ and in the last equality we also used the commutative diagram
\[
\xymatrixcolsep{5pc}
\xymatrix{
X \times \widehat{B_X} \times Y' \times Y \ar[r]^-{p_{123}} \ar[dr]^-{p_{12}} &X \times \widehat{B_X} \times Y' \ar[d] \ar[r]^-{b_X \times {\rm id}_{\widehat{B_X} \times Y'}} &B_X \times \widehat{B_X} \times Y' \ar[d]^-{\overline{p}_{12}} \\
 &X \times \widehat{B_X} \ar[r]^-{b_X \times {\rm id}_{\widehat{B_X}}}  &B_X \times \widehat{B_X}\, .
}
\]
By \cite[Lemma 4.10]{lieblich+olsson:2022},\footnote{In \emph{loc.\! cit.}, the authors assume $\widehat{B_X} = \Pic0 X$. However, their proof works in our more general situation as well (see also \cite[Lemma 3.1]{popa+schnell:invariance}).} one has that 
\begin{equation}\label{LOeq3}
p_{12}^*\cQ_X \otimes p_{14}^*\E \simeq ({\rm id}_X \times \varphi_{\E} \times {\rm id}_{Y' \times Y})^* (p_{42}^*\cP_Y \otimes p_{14}^*\E)\, ,
\end{equation}
where $\cP_Y := (b_Y \times {\rm id}_{\widehat{B_Y}})^* \cP$.
 Therefore, we get
\begin{equation}\label{LOeq4}
{p_{134}}_*(p_{23}^* \cG  \otimes p_{12}^*\cQ_X \otimes p_{14}^*\E) \simeq {p_{134}}_*(( {\rm id}_X \times \varphi_{\E} \times {\rm id}_{Y' \times Y})^*( 
p_{23}^* (\varphi_{\E} \times {\rm id}_{Y'})_* \cG  \otimes 
p_{42}^*\cP_Y \otimes p_{14}^*\E)) 
\end{equation}
and, by arguing as before (in the reverse order) and using the commutativity of \eqref{LOeq1}, we see that the right-hand side in \eqref{LOeq4} is isomorphic 
to $p_{32}^* (b_Y \times {\rm id}_{Y'})^*({u_Y}_*\OO_{Y'}) \otimes p_{13}^* \E$, as desired.
\end{proof}

Let 
\[
\delta_0,\, \delta_1 \colon X_U  \times Y_U  \to X_U  \times U \times Y_U  
\]
be defined as $\delta_0((x, y)) = (x, \psi^{-1}(s_X(x)), y)$ and $\delta_1((x, y)) = (x, s_Y(y), y)$. 
Note that we may (and do) assume that $U = \psi^{-1}(b_X')^{-1}(V)$, where $V \subseteq B_X$ is an open subscheme such that $b_X$ is flat over it.
Consider the commutative diagram
\[
\xymatrixcolsep{3.5pc}
\xymatrix{ 
&X  \times Y' \times Y \ar[r]^-{p_{12}} &X \times Y' \ar[r]^-{b_X \times \mathrm{id}_{Y'}} &B_X \times Y' \\
X_U  \times U \times Y_U  \ar@{^{(}->}[r] \ar@/^1.5pc/[rrr]^-{\alpha_0} & X_U  \times U \times Y \ar[rr] \ar@{^{(}->}[u] & &V \times U\, . \ar@{^{(}->}[u]
}
\] 
Denote by $q_{ij}$ the projections from $X_U  \times U \times Y_U $ onto the $i$-th and $j$-th factors, and by $p_1$ and $p_2$ the two projections from $X_U  \times Y_U $.
Then, the restriction of the left-hand side in \eqref{LOeq0} to $X_U  \times U \times Y_U $ is isomorphic to
\[
\alpha_0^*((\R{u_X}_* \OO_{Y'})|_{V \times U}) \stackrel{\L}{\otimes } q_{13}^* \E_U \simeq \alpha_0^*\R{u_{X, U}}_* \OO_{U} \stackrel{\L}{\otimes } q_{13}^* \E_U
\]
where $u_{X, U} \colon U  \to V \times U$ is the restriction of $u_X$. Consider the cartesian diagram
\[
\xymatrixcolsep{3.5pc}
\xymatrix{ 
X_U  \times Y_U  \ar[r]^-{(\psi^{-1} \circ s_X) \circ p_1} \ar[d]^{\delta_0} &U \ar[d]^{u_{X,U}} \\
X_U  \times U \times Y_U  \ar[r]^-{\alpha_0} & V \times U\, .
}
\] 
By flat base change and projection formula, one has
\[
\alpha_0^*\R{u_{X, U}}_* \OO_{U} \stackrel{\L}{\otimes } q_{13}^* \E_U \simeq \R{\delta_0}_* \OO_{X_U  \times Y_U } \stackrel{\L}{\otimes } q_{13}^*\E_U \simeq \R{\delta_0}_* \delta_0^* q_{13}^*\E_U \simeq \R{\delta_0}_* \E_U. 
\]
Similarly, the restriction of the right-hand side in \eqref{LOeq0} to $X_U  \times U \times Y_U $ is isomorphic to $\R{\delta_1}_* \E_U$. 
In this way, we get an isomorphism
\begin{equation}\label{lc4}
 \R{\delta_0}_* \E_U \xrightarrow{\simeq} \R{\delta_1}_* \E_U.
\end{equation}
Let us also note that 
 the pushforward of \eqref{lc4} through the projection  $q_{13}$ 
\begin{equation}\label{lc5}
\E_U \simeq  \R{q_{13}}_* \R{\delta_0}_* \E_U \to \R{q_{13}}_* \R{\delta_1}_* \E_U  \simeq \E_U
\end{equation}
is the identity morphism of $\E_U$ (\emph{cf}.\! \cite[\S 4.15]{lieblich+olsson:2022}).

Moreover, as  by Claim \ref{rel0} the functor $\Phi_{\E_U}$ is  in particular fully faithful, 
 it holds  true that
${\rm Ext}^i\big(\L i_x^*(\E_U), \L i_x^*(\E_U)\big) = {\rm Ext}^i\big(\Phi_{\E_U}(\mathbb C(x)), \Phi_{\E_U}(\mathbb C(x))\big) = 0$ 
for all $i < 0$ and for any closed point $x \in X_U $, where $i_x \colon  \{x\} \times Y_U  \hookrightarrow X_U  \times Y_U $. Therefore, 
by cohomology and base change for complexes \cite[\S 7.7]{EGA:III}, one has that
  $\R {p_1}_* \R\mathcal{H}om(\E_U, \E_U)$ lies in $\DD^{\geq 0}(X_U )$.  
	
 Hence:
\begin{equation}\label{condabc}
\R{\delta_0}_* \E_U \simeq \R{\delta_1}_* \E_U, \quad \quad   \eqref{lc5}\, \mathrm{holds}, \quad \quad  \R {p_1}_* \R\mathcal{H}om(\E_{U}, \E_{U}) \in \DD^{\geq 0}(X_U ). 
\end{equation}
 At this point the statement follows from \cite[Theorem 1.1]{lieblich+olsson:2022}.
\end{proof}

\section{Irregular  fibrations under derived equivalence}\label{sec:chipos}

We begin by recalling some definitions. We continue to denote by $X$  a   smooth projective complex variety of dimension $n$.
The irregularity of a normal projective variety is defined as  the irregularity of any of its  resolution of singularities.

A \emph{non-singular representative of}  
a  fibration $f	\colon X \to V$ onto a normal projective variety $V$ is  a commutative diagram
\[ \xymatrix{\widetilde X \ar[d]^-{\widetilde f}\ar[r]^{\rho} & X  \ar[d]^-{f}\\
\widetilde{V}\ar[r]^{\theta} & V}  \]
where $\rho$ and $\theta$ are birational morphisms from smooth projective varieties and $\widetilde f$ is a fibration.
We define the abelian subvariety 
\begin{equation*}
\widehat{B_V}  \;  \stackrel{{\rm def} }{=}  \; \rho_*\widetilde f^* \Pic0 \widetilde V \; \subset \; \Pic0 X.
\end{equation*}

%

\begin{theorem}\label{chiposi}
Let $X$ and $Y$ be smooth projective varieties and 
let  $\Phi \colon \DD(X) \to \DD(Y)$ be an equivalence.
If $f\colon X\to V$ is an irregular fibration such that $\widehat{B_V}$ is Rouquier-stable,
then  $Y$ admits a fibration $h\colon Y \to W$ onto a variety $W$ that is birational to $V$. Moreover, 
the general fibers of  $f$ and $h$ are derived equivalent.
\end{theorem}

\begin{proof}
Consider the commutative diagram
\[
\xymatrix{\widetilde X \ar[r]^-{a_{\widetilde X}} \ar[d]_-{\widetilde f} & \Alb \widetilde X \ar[d]^-{\pi} \\
\widetilde V \ar[r]^-{a_{\widetilde V}} &\Alb \widetilde V
}
\]
where $\pi$ is the fibration induced by $\widetilde f$ at the level of   Albanese varieties.
By definition  there is an isomorphism $B_V \cong \Alb \widetilde V$ and $\pi \circ a_X = b_X$, where
 $b_X \colon X \to B_V$ is the Rouquier-stable morphism induced  by $\widehat{B_V} \subset \Pic0 X$ (see \S \ref{sectionRunmixed}).
Since $a_{\widetilde X} = ( a_X \circ \rho) \colon \widetilde X \to X \to \Alb X \simeq \Alb \widetilde X$ and $\rho_*\OO_{\widetilde X} = \OO_X$, we have that 
$X'$ (notations as in \eqref{xy}) is isomorphic to  the base of the Stein factorization of  $\pi \circ a_{\widetilde X} = a_{\widetilde V} \circ \widetilde f$. 
But ${\widetilde f}_* \OO_{\widetilde X} = \OO_{\widetilde V}$, 
hence $X'$ is also  isomorphic to the base of the Stein factorization of $a_{\widetilde V}$, which is birational to $\widetilde V$  as   $a_{\widetilde V}$ is generically finite onto its image. 
Namely, $\widetilde f$ is a non-singular representative of the fibration $s_X \colon X \to X'$ too,
and we have the following commutative diagram
\begin{equation}\label{construction}
\xymatrix{
\widetilde{X} \ar[d]^-{\widetilde{f}} \ar@/^2pc/[rr]^-{a_{\widetilde X}} \ar[r]^-{\rho} &X \ar[d]^-{s_X} \ar[r]^{a_X}  &\Alb \widetilde{X} \ar[d]^-{\pi} \\
 \widetilde{V} \ar@/_2pc/[rr]^-{a_{\widetilde V}} 
\ar[r] &X' \ar[r]^-{b_X'} &\Alb \widetilde{V}\, .
}
\end{equation}
In particular, $X'$ is birational  to $V$.

Let $\widehat{B_Y}:= \varphi (\widehat{B_V})$ and $b_Y \colon Y \to B_Y$ be the corresponding Rouquier-stable morphism.
Moreover, consider the Stein factorization of $b_Y$:
\begin{equation*}
\xymatrix{
 Y \ar@/^2pc/[rr]^-{b_Y} \ar[r]^-{s_Y} & Y' \ar[r] &  B_Y
 } 
\end{equation*}
where the first morphism is a fibration and the second is finite onto its image.
 By Theorem \ref{step1}, there exists an isomorphism 
$X' \simeq Y'$. In particular, by taking $h:= s_Y \colon Y \to W:= Y'$, we have that
$V$ and $W$ are birational. The second statement follows from the above construction and Theorem \ref{step1}(iii).
\end{proof}

 Recall from the Introduction that 
two irregular fibrations $f_1 \colon X \to V_1$  and $f_2 \colon X \to V_2$ are  \emph{equivalent} 
if there exists a birational map 
 $\sigma \colon V_1 \dashrightarrow V_2$ 
such that $f_2=\sigma \circ f_1$.
We record for later use  the following consequence of the proof of  the previous theorem. 
\begin{lemma}\label{lemmaX10}
The irregular fibration $f \colon X \to V$ we started with  is equivalent to $s_X$. In particular, the general fibers of $f$ and $s_X$ are isomorphic.
\end{lemma}

It turns out that irregular fibrations onto varieties of general type provide a natural geometric framework where the Rouquier-stableness
 assumption of Theorem \ref{chiposi} is automatically satisfied. 
\begin{lemma}\label{lemmagentypeRunv}
If $f \colon X \to V$ is an irregular fibration and $V$ is of general type, then the associated abelian variety $\widehat{B_V}$ is Rouquier-stable with respect to any equivalence.
\end{lemma}
\begin{proof}
We aim to prove that $\widehat{B_V}$  is contained in a Rouquier-stable subvariety. 
Take notations as in Subsection \ref{cohsupploci}.
 By Koll\'ar's decomposition theorem \cite{kollar:decomposition} one has that   $\widetilde{f}^*V^0(\widetilde V, \omega_{\widetilde V})  \subseteq V^k(\widetilde X, \omega_{\widetilde X})$, where $k$ is the dimension of the generic fiber of $\widetilde f$ (see \cite[Lemma 6.3]{lombardi:invariants}). Therefore, ${\rho}_*\widetilde{f}^*V^0(\widetilde V, \omega_{\widetilde V})  \subseteq V^k(X, \omega_{X})$.

The Rouquier isomorphism $\varphi$ induces a map
\[
\Pic0\, Y \to \Aut0 X, \quad  \beta \mapsto p_{\Aut0 X} ( \varphi^{-1} ( \rm{id}_Y, \beta ) ) 
\]
whose image is an abelian variety denoted by $A \subseteq \Aut0 X$.
If $A$ is trivial, then $\Pic0 X$ is Rouquier-stable by definition. So we may assume that $\dim A > 0$.
Now take a general point $x_0 \in X$ and consider the orbit map
\[
g \colon A \to X, \quad \xi \mapsto \xi(x_0).
\]
Using Brion's results on the action of  a non-affine algebraic group on  smooth projective varieties
(\cite{brion:affine}, see also \cite[\S 2]{popa+schnell:invariance}),
it can be proved that $V^k (X, \omega_X)$ is contained in the subgroup $\ker \big( g^* \colon \Pic0 X \to \widehat A\big)$ of $\Pic0 X$ (see \cite[p.\ 524, (8)]{lombardi:invariants}). In particular this yields  the inclusion  ${\rho}_* \widetilde{f}^* V^0(\widetilde V, \omega_{\widetilde V}) \subseteq \ker(g^*)$.

By assumption $\widetilde V$ is of maximal Albanese dimension and of general type, 
 hence \cite[Theorem 1]{chen+hacon:pluricanonical} says  that  $V^0(\widetilde V, \omega_{\widetilde V})$ generates $\Pic0 \widetilde V$ as a group. 
Therefore,  from the above discussion we get  
\[
\widehat{B_V} = {\rho}_*\widetilde{f}^*\Pic0 \widetilde V \subseteq \ker(g^*).
\]
Moreover, since $\widehat{B_V}$ is an abelian subvariety, it is actually contained in the connected component $(\ker(g^*))_0$ of $\ker(g^*)$ through the origin. 
Now we employ the fact that   $(\ker(g^*))_0$ is Rouquier-stable as in \cite[p.\ 533]{popa+schnell:invariance}.

\end{proof}

\subsection{Proof of Theorem \ref{Thmfibrintro} $(i)$ and $(ii)$}
The proof of  Theorem \ref{Thmfibrintro} $(i)$ and  $(ii)$  follows 
from Theorem \ref{chiposi} and  Lemma \ref{lemmagentypeRunv}.
\qed

\subsection{Proof of Theorem \ref{Thmfibrintro} $(iii)$} \label{proofkawgen}

 Let $s_X \colon X \to X'$ (\emph{resp.} $s_Y \colon Y \to Y'$) be the fibration induced by $\widehat{B_V}$  
(\emph{resp.} $\varphi (\widehat{B_V})$). 
We know that $$W \; :=\; Y^\prime \; \simeq \; X^\prime$$ thanks to Theorem \ref{step1} and, moreover,  
\begin{equation}\label{comkaw}
\Supp (\E)\subseteq X \times_W Y,
\end{equation}
where $\E \in \DD( X \times Y)$ is the kernel of the   equivalence.
At this point, the proof is a relative version of Kawamata's technique \cite{kawamata:dequivalence}.
Let $Z \subseteq \Supp (\E)$ be an irreducible component  such that the first projection $\pi_X \colon Z \rightarrow X$ is surjective  (see \cite[Corollary 6.5]{huybrechts:fouriermukai}).
In particular, the inequality  $\dim X \leq \dim Z$ holds. Denote by $\pi_Y \colon Z \rightarrow Y$ the second projection.
From \eqref{comkaw} we get a \emph{commutative} diagram
\begin{equation}\label{relK00}
\xymatrix{
     &Z \ar[dl]_-{\pi_X} \ar[dr]^-{\pi_Y}  & \\
X \ar[dr]^-{s_X} & &Y \ar[dl]_-{s_Y} \\
 &W &		
}
\end{equation}
Note that, for
 any point $w \in W$, one has 
\[
\pi_X^{-1}(s_X^{-1}(w)) = Z \cap (s_X^{-1}(w) \times s_Y^{-1}(w)) \subseteq \Supp(\E) \cap (s_X^{-1}(w) \times s_Y^{-1}(w)) = \Supp( \mathbf L\iota^*\E)
\]
where $\iota \colon s_X^{-1}(w) \times s_Y^{-1}(w) \hookrightarrow X \times Y$ is the inclusion map 
(the last equality is \cite[Lemma 3.29]{huybrechts:fouriermukai}).
Thanks to Lemma \ref{lemmaX10}, for a general $w \in W$ the (anti)canonical bundle of $s_X^{-1}(w)$ is big, which implies, by an argument of Kawamata in   \cite{kawamata:dequivalence}  that the morphism $\pi_Y^{-1}(s_Y^{-1}(w)) \xrightarrow{\pi_Y} s_Y^{-1}(w)$ is generically finite.
We briefly sketch this for reader's convenience. Set  $X_w := s_X^{-1}(w)$,  $Y_w := s_Y^{-1}(w)$, and $Z_w := Z \cap (s_X^{-1}(w) \times s_Y^{-1}(w))$ for a general point $w \in W$. Moreover, let $\nu_w \colon \widetilde{Z_w} \rightarrow Z_w$ be the normalization and 
assume that $\omega_{X_w}$ is big (the other case is completely analogue). By Kodaira's lemma, for $m \gg 0$ one has that
$\omega_{X_w}^{\otimes m} \simeq  \OO_{X_w}(H) \otimes \OO_{X_w}(D)$, where $H$ is an ample divisor and $D$ is an effective divisor on $X_w$.
We now prove that the morphism 
$(\pi_Y \circ \nu_w) \colon \widetilde{Z_w} \setminus \nu_w^{-1}\pi_X^{-1}(D) \to Y_w$ is finite. Suppose by contradiction that there exists an irreducible curve $C \subseteq \widetilde{Z_w}$ contracted by $\pi_Y \circ \nu_w$ and such that $C \not\subset \nu_w^{-1}\pi_X^{-1}(D)$. 
Then $$0 \; = \; \deg \big( \nu_w^{*}\pi_Y^{*}(\omega_{Y_w}) \big) \big|_C\;  = \;  \deg \big( \nu_w^{*}\pi_X^{*}(\omega_{X_w}) \big) \big|_C 
\; \geq \; \frac{1}{m} \deg \big( \nu_w^{*}\pi_X^{*}\OO(H) \big) \big|_C \; > \;  0,$$ where the second equality is  an application of \cite[Lemma 6.6]{huybrechts:fouriermukai}. 

In particular, $\pi_Y \colon Z \rightarrow Y$ is generically finite and $\dim Z \leq \dim Y$. But we already know that 
\[
\dim Y = \dim X \leq \dim Z.
\]
Therefore, $\dim X = \dim Z$. At this point another well-established argument due to Kawamata \cite{kawamata:dequivalence} says that $X$ and $Y$ are $K$-equivalent (see also \cite[p.\ 149]{huybrechts:fouriermukai}, or \cite[Lemma 15]{lombardipopa:conj}).
This concludes the proof of $(iii)$.

The argument we just employed
 also provides a further generalization of Kawamata's birational reconstruction theorem (see the Introduction).
\begin{theorem}\label{relkawintro}
Let $\Phi \colon \DD(X) \to \DD(Y)$ be an equivalence and let $\widehat{B_X} \subset \Pic0 X$ be a
 Rouquier-stable subvariety. 
If the (anti)canonical line bundle of the general fiber of the Rouquier-stable morphism $b_X$ is big, 
 then $X$ and $Y$ are $K$-equivalent. 
\end{theorem}

\begin{remark} If $\omega_X$ (\emph{resp}. $\omega_X^{-1}$) is big as in Kawamata's theorem, then $\Aut0 X$ is an affine algebraic group (see Subsection \ref{subs:aut}). Hence the whole 
$\Pic0 X$ is 
 Rouquier-stable  and   $\omega_X$ (\emph{resp.} $\omega_X^{-1}$) is obviously  $b_X$-big (note that $b_X = a_X$ if $\widehat{B_X} = \Pic0 X$).
Moreover, as recalled in Subsection \ref{subs:aut}, varieties with non-zero Euler characteristic have affine automorphism group $\Aut0(-)$. Hence Corollary \ref{corintro} of the Introduction is a particular case of the above Theorem \ref{relkawintro}.
\end{remark}

\subsection{Proof of Theorem \ref{introbij}}

An  \emph{irregular} $k$\emph{-fibration} is an irregular fibration  onto a variety of dimension $k$.
For any  variety $X$ and 
integer $0< k< n := 	\dim X$
 we  define the following set:
$$G_X  \stackrel{{\rm  def}}{=}  \{ \, 
  \mbox{equivalence classes of irregular }k\mbox{-fibrations }f\colon X \rightarrow V$$ $$ \mbox{ such that   $V$ is of general type and 
  $0< k< \dim X$}\, \}.
$$

We aim to prove Theorem \ref{introbij}. Indeed, we prove the more precise Theorem \ref{thm:corr} below.
\begin{theorem}\label{thm:corr}
Let $\Phi \colon \DD(X) \to \DD(Y)$ be a derived equivalence.
There exists a bijective  correspondence $\mu_{\Phi} \colon G_X \to G_Y$ such that if 
$\mu_{\Phi}(f\colon X \to V) = ( h \colon Y \to W )$, then the varieties $V$ and $W$ are birational. 
Moreover, the generic fibers of $f$ and $h$ are derived equivalent.
\end{theorem}
In the rest of this section we prove the above theorem.
The function $\mu_{\Phi}$ is defined by Theorem \ref{chiposi}:  
we take the Stein factorization of the Rouquier-stable morphism     $b_Y$
$$b_Y \colon Y \;  \xrightarrow{h := s_Y} \; W := Y'  \; \stackrel{b'_Y}{\longrightarrow}  \;   B_Y,$$
where $\widehat{B_Y} := \varphi ( \widehat{B_V})$. 
In particular, we already know that $V$ and $W$ are birational 
and that  the generic fibers of $f$ and $h$ are derived equivalent.

Now we turn to prove that $\mu_{\Phi}$ is a bijection. Take notations as in the proof of Theorem \ref{chiposi}.
By Theorem \ref{step1}(ii) and \eqref{construction},
there exists an isomorphism of varieties $\psi \colon W \stackrel{\sim}{\longrightarrow} X'$ such that the diagram
\[
\xymatrix{
\widetilde{V}  \ar@/^2pc/[rr]^-{a_{\widetilde V}} \ar[r] &X'  \ar[r]^-{b'_X}  &B_V \simeq \Alb \widetilde{V}  \\
 \widetilde{W} \ar@{-->}[u] \ar[r] &W \ar[u]^-{\psi} \ar[r]^-{b'_Y} &B_Y \ar[u]^-{\widehat{\varphi}}
}
\]
 is commutative. Hence we get that $B_Y \simeq \Alb \widetilde W$, and moreover that 
 the bottom composition is isomorphic to the Albanese map $a_{\widetilde W}$ of a resolution $\widetilde W$ of $W$.

Now let 
\[ \xymatrix{\widetilde Y \ar[d]^-{\widetilde h}\ar[r]^-{\sigma} & Y  \ar[d]^-{h} \ar[r] &\Alb  Y \ar[d] \\
\widetilde{W}\ar[r]^-{\xi} & W \ar[r] & \Alb \widetilde W \simeq B_Y}  \]
where  the left vertical morphism is a non-singular representative of the irregular fibration $h$, and the right-hand one is the fibration induced by the universal property of the Albanese variety.
Then  we have
$\widehat{B_Y} \simeq \sigma_* \widetilde{h}^* \Pic0 \widetilde W =: \widehat{B_W} \subseteq \Pic0 Y$.
Hence
\[
\varphi (\widehat{B_V}) = \widehat{B_Y} \simeq \widehat{B_W}.
\]
At this point, if we apply $\mu_{\Phi^{-1}}$ to $h$ where $\Phi^{-1}$ is a quasi-inverse of $\Phi$, 
we get $$\mu_{\Phi^{-1}}(h) = \mu_{\Phi^{-1}}(\mu_{\Phi}(f)) = s_X$$ thanks to the functoriality of the Rouquier isomorphism. 
By Lemma \ref{lemmaX10} $f$ and $s_X$ are equivalent fibrations of $X$, so
$\mu_{\Phi^{-1}} \circ \mu_{\Phi} = {\rm id}_{G_X}$ and,
since the role of $X$ and $Y$ can be symmetrically exchanged, we also get $\mu_{\Phi} \circ \mu_{\Phi^{-1}}  = {\rm id}_{G_Y}
$ by the same reasoning.
 
\begin{remark}\label{lastrmk}
Let us restrict ourselves  
to irregular fibrations $f \colon X \to V$ onto varieties $V$  admitting a  morphism $c_V \colon V \to \Alb \widetilde V$  finite onto its image and such that the composition $\widetilde V \xrightarrow{\rho} V \xrightarrow{c_V} \Alb \widetilde V$ equals the Albanese map of a desingularization $\widetilde V$.\footnote{This is precisely what happens if $\dim V = 1$.} In this case two 
fibrations $f\colon X \to V$ and $f' \colon X \to V'$ are equivalent if there exists an isomorphism $\sigma \colon V \to V'$ such 
that $f' = \sigma \circ f$. Then the   bijection of Theorem \ref{thm:corr} is base-preserving in a stronger sense: namely, $V$ is isomorphic to $W$. In fact, there exists an isomorphism $\sigma \colon V \xrightarrow{\sim} X'$ such that $s_X = \sigma \circ f$ (see \cite[Lemma 19]{lombardi:fibrations}).    
\end{remark}

\subsection{Proof of Proposition \ref{Propfibrintro}}
Once the $K$-equivalence among $X$ and $Y$ has been proved by Theorem \ref{Thmfibrintro} (iii),  Proposition \ref{Propfibrintro} follows at once by standard arguments (see \cite[p.\ 304]{lombardipopa:conj}).  Namely,
if the  rational map 
$\psi \colon Y  \dashrightarrow X$ induced by the $K$-equivalence is not a morphism, there exists a curve 
$C \subseteq Z$ that is contracted by $\pi_Y$ but not by $\pi_X$ (see \eqref{relK00}). So $(\pi_X^* \omega_X \cdot C) = (\pi_Y^* \omega_Y \cdot C) = 0$.
On the other hand, by  Lemma \ref{lemmaX10} and the above Remark \ref{lastrmk}, we see that 
 $\omega_X$ is $f$-(anti)ample if and only if it is $s_X$-(anti)ample, and, since 
 $\pi_X(C)$ is contained in a fiber of $s_X$, one gets $(\omega_X \cdot \pi_X(C) ) \neq 0$. 
This gives a contradiction. 
Hence $\psi$ is a crepant birational morphism between smooth projective varieties, hence an 
 isomorphism.

\subsection{Proof of Theorem \ref{thmnotgentype}}

For the proof of Theorem \ref{thmnotgentype} we apply the main  result of \cite{chen+hacon:pluricanonical}.
Let $f \colon X \to V$ be an irregular fibration of $X$. Since by definition $V$ is of  maximal Albanese dimension, one has that $\mathrm{kod}(V) \geq 0$. Then it makes sense to consider the Iitaka fibration  of $V$, which by definition is the Iitaka fibration of a non-singular model of $V$.
So we get the following commutative diagram
\begin{equation}\label{fincomm}
\xymatrix{ \widetilde X \ar[r]^-{\rho}  \ar[d]^-{\widetilde f} &X \ar[r]^-{a_X} &\Alb X \ar[d]^-{\pi} \\
\widetilde V \ar[rr]^-{a_{\widetilde V}} \ar[d]^-{g} & &B_V = \Alb \widetilde V \ar[d]^-{\kappa} \\
Z_V \ar[rr]^-{a_{Z_V}} & &\Alb Z_V 
}
\end{equation}
where $Z_V$ is a smooth projective variety of dimension $\dim Z_V = \mathrm{kod}(V)$, and $\kappa$ is the fibration 
between  Albanese varieties induced by the Iitaka fibration $g$ of $V$.

By \cite[Theorem 2.3]{chen+hacon:pluricanonical}, the abelian variety 
$\widehat{\kappa}(\Pic0 Z_V)$ is contained in  the abelian subvariety of $\Pic0 \widetilde V$ generated by $V^0(\widetilde V, \omega_{\widetilde V})$.\footnote{For the sake of clarity, let us say that we are applying \cite[Theorem 2.3]{chen+hacon:pluricanonical} to the generically finite morphism $a_{\widetilde V}$, and the variety $\Pic0 S$ in \emph{loc.\! cit.\!} coincides with our $\widehat{\kappa}(\Pic0 Z_V)$.} 
Since in the proof of Lemma \ref{lemmagentypeRunv} we  observed that
${\rho}_* \widetilde{f}^* V^0(\widetilde V, \omega_{\widetilde V})$ is contained in a subgroup of $\Pic0 X$ whose connected component through the origin is a Rouquier-stable subvariety,
it follows  from the commutativity of \eqref{fincomm}  that 
$\widehat{\pi} \big( \widehat{\kappa}(\Pic0 Z_V) \big)$ is a Rouquier-stable subvariety of $\Pic0 X$.  
Hence, by taking the Stein factorization of the morphism induced by $\varphi \big(\widehat{\pi} \big(\widehat{\kappa}(\Pic0 Z_V) \big)  \big) 
\subseteq \Pic0 Y$, we obtain a 
 fibration $h \colon Y \to W$.
 Since the base of the Stein factorization of the composition $\kappa \circ \pi \circ a_X$ is equal to the base of the Stein factorization of $a_{Z_V}$, which 
 is generically finite onto its image \cite[Proposition 2.1(a)]{haconpardini:albanese}, we see that $W$ is birational to $Z_V$ by Theorem \ref{step1}(i).

\bibliographystyle{amsalpha}
\bibliography{bibl}

\end{document}